\documentclass[11pt]{article}

\usepackage{amsthm,amsmath,amssymb,graphics,graphicx,hyperref,epstopdf}

\title{Toric Ideals of Lattice Path Matroids and Polymatroids}
\author{Jay Schweig \footnote{University of Kansas; jschweig@math.ku.edu}}
\newtheorem{theorem}{Theorem}[section]
\newtheorem{proposition}[theorem]{Proposition}

\newtheorem{lemma}[theorem]{Lemma}
\newtheorem{definition}[theorem]{Definition}

\newtheorem{conjecture}[theorem]{Conjecture}

\newtheorem{example}[theorem]{Example}

\newcommand{\G}{\mathcal{G}}

\newcommand{\ssection}[1]{%
  \section[#1]{\centering\normalfont\scshape #1}}
\newcommand{\ssubsection}[1]{%
  \subsection[#1]{\raggedright\normalfont\itshape #1}}

\begin{document} 

\maketitle

\begin{abstract}
We show that the toric ideal of a lattice path polymatroid is generated by quadrics corresponding to symmetric exchanges, and give a monomial order under which these quadrics form a Gr\"obner basis.  We then obtain an analogous result for lattice path matroids.  
\end{abstract}

\ssection{Introduction}
\label{intro}
If $B = \{i_1, i_2, \ldots, i_r\}$ is a basis of a matroid $M$, the toric map of $M$ sends the base ring variable $Y_B$ to the monomial $x_{i_1}x_{i_2} \cdots x_{i_r}$ and is naturally extended over all polynomials of variables indexed by bases of $M$.  White has conjectured \cite{white} that the kernel of this map is generated by quadratic binomials corresponding to symmetric exchanges between pairs of bases of $M$.  

Lattice path matroids, introduced by Bonin, de Mier and Noy in \cite{bdn} and studied further in \cite{bdn2}, are an especially nice class of matroids whose bases are in correspondence with certain planar lattice paths.  Subclasses of these matroids appeared in \cite{carly} and \cite{federico}.  In \cite{me}, the study of enumerative properties of such matroids gave rise to a related class of discrete polymatroids, in the sense of Herzog and Hibi \cite{hh}, known as lattice path polymatroids.

As in \cite{hh}, toric maps can be defined for discrete polymatroids as well, inspiring a generalization of White's conjecture.  In Theorem \ref{main}, we show that White's conjecture holds for lattice path polymatroids.  We also provide a monomial order under which the generating set of symmetric exchange binomials forms a Gr\"obner basis.  We then show how a lattice path matroid can be realized as the set of squarefree bases of a related lattice path polymatroid, allowing us to prove an analogue of Theorem \ref{main} for such matroids.  

\ssection{Preliminaries}
\label{sec:1}

We assume the reader has a basic knowledge of matroid theory (see \cite{oxley}).  All our monomials are in the variables $\{x_0, x_1, x_2, \ldots\}$.  When $m$ is a monomial, we write $d_i(m)$ to mean the degree of $x_i$ in $m$.   

\begin{definition}
Let $\Gamma$ be a finite collection of monomials.  Then $\Gamma$ is a \emph{discrete polymatroid} if it satisfies the following two properties:  1) If $m\in \Gamma$ and $m'$ divides $m$, then $m' \in \Gamma$, and 2) If $m, m' \in \Gamma$ and the degree of $m$ is greater than that of $m'$, there exists an index $i$ such that $d_i(m) > d_i(m')$ and $x_im' \in \Gamma$. 
\end{definition}

Thus, a matroid can be viewed as a squarefree discrete polymatroid.  It is easily seen that the maximal monomials of a discrete polymatroid $\Gamma$ are all of the same degree.   In keeping with standard matroid terminology, we refer to these maximal monomials as \emph{bases}, and we say their degree is the \emph{rank} of $\Gamma$.  Discrete polymatroids were introduced by Herzog and Hibi in \cite{hh}, where the following polymatroidal analogue of the classical symmetric exchange property for matroids was proven.

\begin{proposition}
Let $m$ and $m'$ be bases of a discrete polymatroid $\Gamma$, and choose $i$ with $d_i(m) > d_i(m')$.  Then there exists an index $j$ with $d_j(m) < d_j(m')$, such that both $\frac{x_j}{x_i} m$ and $\frac{x_i}{x_j} m'$ are bases of $\Gamma$.
\end{proposition}

In the case of the above proposition, we say that the bases $\frac{x_j}{x_i} m$ and $\frac{x_i}{x_j} m'$ are obtained from $m$ and $m'$ via a \emph{symmetric exchange}.  

\ssubsection{Lattice paths}

Fix two integers $n, r >0$.  For our purposes, a \emph{lattice path} is a sequence of unit-length steps in the plane, each either due north or east, beginning at the origin and ending at the point $(n, r)$.  For a lattice path $\sigma$, define a set $N(\sigma) \subseteq [n+r]$ by the following rule: $i\in N(\sigma) \Leftrightarrow$ the $i^{th}$ step of $\sigma$ is north.

Let $\sigma$ and $\tau$ be lattice paths.  We say that $\sigma$ is above $\tau$ if for all $i \leq n$ the $i^{th}$ east step of $\sigma$ lies on or above the $i^{th}$ east step of $\tau$.  In this case, we write $\sigma \succeq \tau$.   

Now fix two lattice paths $\alpha$ and $\omega$ with $\alpha \succeq \omega$.  

\begin{theorem}
\cite{bdn} The collection $\{N(\sigma): \alpha \succeq \sigma \succeq \omega\}$ is the set of bases of a matroid.
\end{theorem}

We write $\mathcal{M}(\alpha, \omega)$ to denote the matroid determined by the paths $\alpha$ and $\omega$.  Matroids arising in this fashion are known as \emph{lattice path matroids}.  For any lattice path $\sigma$, define a monomial $m(\sigma)$ by the following rule: the degree of $x_i$ in $m(\sigma)$ is the number of north steps taken by $\sigma$ along the vertical line $x = i$.  

\begin{theorem}
\cite{me} The collection $\{ m(\sigma): \alpha \succeq \sigma \succeq \omega\}$ is the set of bases of a discrete polymatroid.
\end{theorem}

\begin{figure}[htp]
\centering
\includegraphics[height=1.1in]{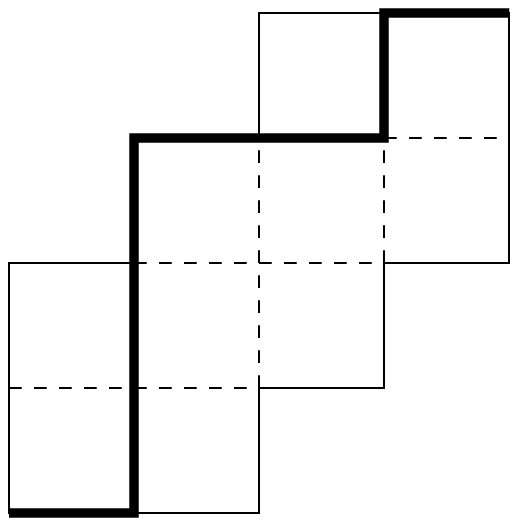}
\caption{The lattice path matroid $\mathcal{M}(\alpha, \omega)$ where $N(\alpha) = \{1, 2, 4, 6\}$ and $N(\omega) = \{3, 5, 7, 8\}$.  If $\sigma$ is the bold path, $m(\sigma) = x_1^3x_3$.}\label{path}
\end{figure}

We call such discrete polymatroids \emph{lattice path polymatroids}, and write $\Gamma(\alpha, \omega)$ to denote the polymatroid determined by $\alpha$ and $\omega$.  For a lattice path $\sigma$ whose last step is east, let $\sigma^+$ be the path obtained from $\sigma$ by removing its last east step, and adding an east step at the beginning.  It is easily seen that the lattice path matroid $\mathcal{M}(\alpha,\omega)$ is coloop-free if and only if $\alpha$ and $\omega$ share no north steps (and thus the last step of $\alpha$ is east).  That is, $\mathcal{M}(\alpha, \omega)$ is coloop-free if and only if $\alpha^+ \succeq \omega$.  The following theorem motivated the introduction of lattice path polymatroids.  

\begin{theorem}
\cite{me} Suppose the lattice path matroid $\mathcal{M}(\alpha, \omega)$ is coloop-free.  Then its h-vector is the f-vector (or degree sequence) of $\Gamma(\alpha^+, \omega)$.  
\end{theorem}

\begin{example}
If $\alpha$ is the path consisting of $r$ north steps followed by $n$ east steps and $\omega$ is any other path, then $\mathcal{M}(\alpha, \omega)$ is a \emph{shifted} matroid.  In \cite{carly}, it is shown that every shifted matroid is of this form.  In this case, bases of the polymatroid $\Gamma(\alpha, \omega)$ are generators of the smallest Borel-fixed ideal containing $m(\omega)$ (see \cite{fms}).  
\end{example}

\ssubsection{Toric ideals}

The \emph{base ring} of a polymatroid $\Gamma$ is the polynomial ring $\mathbb{C}[Y_m: m$ is a basis of $\Gamma]$.  If $n$ and $n'$ are obtained from $m$ and $m'$ by a symmetric exchange, we call $Y_mY_{m'} - Y_nY_{n'}$ a \emph{symmetric exchange binomial}.  The \emph{toric ideal} of $\Gamma$ is the kernel of the map $\phi :  \mathbb{C}[Y_m: m$ is a basis of $\Gamma] \rightarrow \mathbb{C}[x_0, x_1, x_2, \ldots]$ defined by 
\[
\phi(Y_{m_1}Y_{m_2} \cdots Y_{m_t}) = m_1m_2 \cdots m_t
\]
and extended by linearity.  Clearly, any symmetric exchange binomial is in the toric ideal of $\Gamma$.  

\begin{conjecture}[White's conjecture, adapted for polymatroids]\label{white}
The toric ideal of $\Gamma$ is generated by symmetric exchange binomials.
\end{conjecture}

For a set $V = \{m_1, m_2, \ldots, m_t\}$ of bases of $\Gamma$, write $M_V$ as short for the base ring monomial $Y_{m_1}Y_{m_2} \cdots Y_{m_t}$.  Now for any monomial $\mu$ of degree $>r$, we define a simple graph $\mathcal{G}(\mu)$, known as a \emph{symmetric exchange graph}, as follows.  The vertices of $\mathcal{G}(\mu)$ are all sets $ V = \{m_1, m_2, \ldots, m_t\}$ of bases of $\Gamma$ with $\phi(M_V) = \mu$ (that is, $m_1 m_2 \cdots m_t   = \mu$), and two vertices $V$ and $W$ are connected by an edge whenever $M_V - M_W = NS$ for some monomial $N$ and symmetric exchange binomial $S$.  Put another way, $V$ and $W$ are connected by an edge if $W$ can be obtained from $V$ by performing a symmetric exchange on two of its constituent bases.  Although $\mathcal{G}(\mu)$ depends on $\Gamma$, we omit this from the notation whenever it will be clear from context.

The following was inspired by Blasiak's techniques in \cite{blasiak}, where Conjecture \ref{white} was proven for graphic matroids. 

\begin{theorem}\label{blasiak} 
Suppose that $\mathcal{G}(\mu)$ is connected for any monomial $\mu$ of degree $>r$.  Then Conjecture \ref{white} holds for $\Gamma$.
\end{theorem}

\begin{proof}
Any polynomial in the toric ideal of $\Gamma$ can be written as a sum of binomials of the form $M_V - M_W$, where $V$ and $W$ are vertices of some $\mathcal{G}(\mu)$.  Since $\mathcal{G}(\mu)$ is connected, there exists a path $V = V_0, V_1, V_2, \ldots, V_k = W$ where each $V_i$ and $V_{i+1}$ are connected by an edge.  Now write 
\[
M_V - M_W = (M_V - M_{V_1}) + (M_{V_1} - M_{V_2}) + (M_{V_2} - M_{V_3}) + \cdots + (M_{V_{k-1}} - M_W).
\]
Since each parenthesized term in this sum is the product of a monomial with a symmetric exchange binomial, the result follows.
\end{proof}

\ssubsection{Gr\"obner bases}

Our treatment here is brief; the reader unfamiliar with the theory of Gr\"obner bases is referred to \cite{eisenbud}.  

Let $>_\ell$ be a total order on monomials of the base ring of a polymatroid $\Gamma$ with $M >_\ell 1$ for any monomial $M \neq 1$.  The order $>_\ell$ is called a \emph{monomial order} if $M >_\ell M'$ implies that $MN >_\ell M'N$ for any monomials $M, M',$ and $N$.  

If $>_\ell$ is a monomial order on the base ring and $\mu$ is a monomial, define a directed graph $\mathcal{G}^\ell(\mu)$ as follows: the vertices and edges are those of $\mathcal{G}(\mu)$.  If $V$ and $W$ are vertices of $\mathcal{G}(\mu)$ joined by an edge, direct the corresponding edge of $\mathcal{G}^\ell(\mu)$ towards $W$ if $M_V >_\ell M_W$ and towards $V$ if $M_W >_\ell M_V$.  Note that $\mathcal{G}^\ell(\mu)$ is acyclic, since $>_\ell$ is a total order.

The following lemma, whose straightforward proof we omit, is an elementary result from graph theory.

\begin{lemma}\label{uniquesink}
Let $G$ be a finite and acyclic directed graph, and suppose $G$ has a unique sink $v_0$.  Then for any vertex $w$ of $G$, there exists a directed path from $w$ to $v_0$.
\end{lemma}

\begin{theorem}\label{Grobnerswap}
Let $>_\ell$ be a monomial order on the base ring of $\Gamma$, and suppose that $\mathcal{G}^\ell(\mu)$ has a unique sink anytime it is nonempty.  Then Conjecture \ref{white} holds for $\Gamma$ and the symmetric exchange binomials, under the order $>_\ell$, form a Gr\"obner basis for its toric ideal.
\end{theorem}

\begin{proof}
To see that Conjecture \ref{white} holds for $\Gamma$, note that Lemma \ref{uniquesink} implies that any two vertices of some $\mathcal{G}(\mu)$ are in the same connected component (since $\mathcal{G}(\mu)$ is just $\mathcal{G}^\ell(\mu)$ with the edge orientations removed).  Therefore every $\mathcal{G}(\mu)$ is connected, and Theorem \ref{blasiak} gives us that the toric ideal of $\Gamma$ is generated by symmetric exchange binomials.

To finish the proof, we apply Buchberger's algorithm (again, see \cite{eisenbud}) to the set of symmetric exchange binomials.  The S-pair of two symmetric exchange binomials can be represented as $M_V - M_W$, for two vertices $V, W$ of some $\mathcal{G}^\ell(\mu)$.  A step in the reduction of this binomial with respect to the set of symmetric exchange binomials consists either of replacing $M_V$ with $M_{V'}$ where $V \rightarrow V'$ is a directed edge of $\mathcal{G}^\ell(\mu)$ or of replacing $M_W$ with $M_{W'}$ where $W \rightarrow W'$ is a directed edge of $\mathcal{G}^\ell(\mu)$.  Let $V_0$ be the unique sink of $\mathcal{G}^\ell(\mu)$.  By Lemma \ref{uniquesink}, this binomial reduces to $M_{V_0} - M_{V_0} = 0$.
\end{proof}

\ssection{Lattice path polymatroids}

For the remainder of this paper, fix $n$ and $r$ and let $\alpha$ and $\omega$ be two lattice paths to $(n, r)$ with $\alpha \succeq \omega$.  To eliminate excess notation we often identify a path $\sigma$ with the associated monomial $m(\sigma)$, writing, for example, $Y_\sigma$ rather than $Y_{m(\sigma)}$ and $d_i(\sigma)$ rather than $d_i(m(\sigma))$.  This section is devoted to proving the following theorem.

\begin{theorem}\label{main}
Let $\Gamma = \Gamma(\alpha, \omega)$ be a lattice path polymatroid.  Then the toric ideal of $\Gamma$ is generated by symmetric exchange binomials.  Moreover, there exists a monomial order on the base ring of $\Gamma$ under which the symmetric exchange binomials form a Gr\"obner basis for the toric ideal.
\end{theorem}

First, we build a monomial order on the base ring of a lattice path polymatroid $\Gamma$ so that we may apply Theorem \ref{Grobnerswap}.

For any lattice path $\sigma$, define $\ell(\sigma)$ to be the following $nr$-tuple:
\[
(\ell_{0, r}, \ell_{0, r-1}, \ldots, \ell_{0, 1}, \ell_{1, r}, \ell_{1, r-1}, \ldots, \ell_{1,1}, \ldots, \ell_{n-1, r},  \ell_{n-1, r-1}, \ldots, \ell_{n-1, 1})
\]
where $\ell_{i, j} = 1$ if the topmost north step of $\sigma$ along the line $x = i$ goes from $(i, j-1)$ to $(i, j)$, and $\ell_{i, j} = 0$ otherwise.  For a base ring monomial $M =Y_{\sigma_1}Y_{\sigma_2} \cdots Y_{\sigma_t}$, let $\ell(M) = \sum_{1 \leq i \leq t} \ell(\sigma_i)$, where the sum of vectors is taken componentwise.  

\begin{figure}[htp]
\centering
\includegraphics[height = 1.1in]{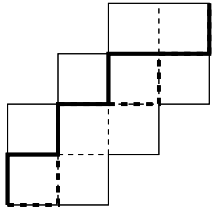}
\caption{If  $n = r = 4$ and $\sigma_1$ and $\sigma_2$ are the paths above, $\ell(Y_{\sigma_1}Y_{\sigma_2}) = (0,0,0,1,0, 0, 2, 0, 0, 1, 0, 0, 0, 1, 0, 0)$.}\label{path}
\end{figure}

Now let $M$ and $M'$ be base ring monomials, and write $M >_{\ell} M'$ whenever $\ell(M)$ lexicographically precedes $\ell(M')$.  Note that $>_\ell$ is not yet a \emph{total} order, as clearly there may be monomials $M \neq M'$ with $\ell(M) = \ell(M')$.  

Indeed, let $M = Y_{\sigma_1} Y_{\sigma_2} \cdots Y_{\sigma_t}$ and $M' = Y_{\tau_1} Y_{\tau_2} \cdots  Y_{\tau_t}$ be two distinct base ring monomials with $\ell(M) = \ell(M')$, where the indexing paths of these monomials are ordered so that whenever $i < j$, $\ell(\sigma_i)$ lexicographically precedes $\ell(\sigma_j)$ and $\ell(\tau_i)$ precedes $\ell(\tau_j)$.  Extend the definition of $>_\ell$ to say that $M >_\ell M'$ if $\ell(\tau_i)$ lexicographically precedes $\ell(\sigma_i)$ for the least $i$ such that $\sigma_i \neq \tau_i$.

Since a path $\sigma$ is clearly determined by the vector $\ell(\sigma)$, this completes $>_\ell$ to a total order on all monomials in the base ring (once we set $M >_\ell 1$ for any monomial $M$).  Moreover, if $M >_\ell M'$, then $MN >_\ell M'N$, since $\ell(MN) = \ell(M) + \ell(N)$ and $\ell(M'N) = \ell(M') + \ell(N)$.  Thus $>_\ell$ is a monomial order.

\begin{definition}
Let $V = \{\sigma_1, \sigma_2, \ldots, \sigma_t\}$ be a vertex of $\mathcal{G}(\mu)$ for some $\mu$, where again we have written $\sigma_i$ to mean $m(\sigma_i)$.  We call the vertex $V$ \emph{thin} if it has the following two properties: 
\begin{itemize}
\item[1:] For any two paths $\sigma_i, \sigma_j \in V$, either $\sigma_i \succeq \sigma_j$ or $\sigma_j \succeq \sigma_i$.  
\item[2:] For any $i$, the $i^{th}$ east steps of any two paths in $V$ are at most a unit length apart.  

\end{itemize}
\end{definition}

\begin{figure}[htp]
\centering
\includegraphics[height=1.1in]{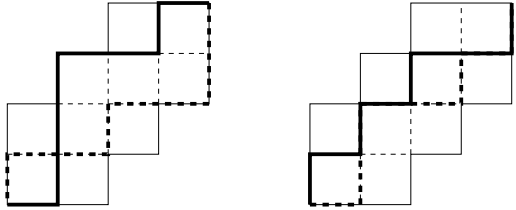}
\caption{Two vertices of the graph $\mathcal{G}(x_0x_1^3x_2x_3x_4^2)$.  The second is thin, while the first is not.}\label{thin}
\end{figure}

Thin vertices, as shown by Proposition \ref{switch} and Lemma \ref{uniquethin}, will be sinks in the directed graphs $\G^\ell(\mu)$.  

\begin{proposition}\label{switch}
Let $V$ be a vertex of some $\G(\mu)$ that is not thin.  Then there is a vertex $V'$ of $\G(\mu)$ resulting from a symmetric exchange between two bases in $V$ such that $M_V >_\ell M_{V'}$.  In other words, $V \rightarrow V'$ is a directed edge of $\G^\ell(\mu)$.   
\end{proposition}

\begin{proof}
Let $V = \{\sigma_1, \sigma_2, \ldots, \sigma_t\}$.  We handle two cases, each corresponding to a way in which a vertex may fail to be thin.  

First suppose that the $(i+1)^{st}$ east step of some path in $V$ is more than a unit length above the $(i+1)^{st}$ east step of another path in $V$, and let $i$ be minimal with this property.  Let $\sigma_p$ be the path with the highest $(i+1)^{st}$ east step, and let $\sigma_q$ be the path with the lowest.  By the minimality of $i$, $d_i(\sigma_p) > d_i(\sigma_q)$.  Since the two paths eventually meet, there must be some $j > i$ such that $d_{j}(\sigma_p) < d_{j}(\sigma_q)$.  Let $j$ be minimal with this property, let $\sigma_p'$ be the path obtained from $\sigma_p$ by adding a north step along $x = j$ and removing one along $x = i$, and let $\sigma_q'$ be the path obtained from $\sigma_q$ by adding a north step along $x = i$ and removing one along $x = j$.  Note that $\sigma_p'$ and $\sigma_q'$ are the results of a symmetric exchange between $\sigma_p$ and $\sigma_q$, although we still need to show that both $\sigma_p'$ and $\sigma_q'$ are paths in $\Gamma(\alpha, \omega)$.  To see this, note that the minimality of $j$ implies that every east step of $\sigma_p$ between $x = i $ and $x = j$ is \emph{strictly} above the corresponding east step of $\sigma_q$.  Thus $\sigma_p \succeq \sigma_p' \succeq \sigma_q$ and $\sigma_p \succeq \sigma_q'\succeq \sigma_q$, meaning both $\sigma_p'$ and $\sigma_q'$ are between $\alpha$ and $\omega$.  Let $V'$ be the vertex resulting from this symmetric exchange.  Then $V'$ is identical to $V$ to the left of the line $x = i$.  Since neither $\sigma_p'$ nor $\sigma_q'$ attains the same height on the line $x = i$ as $\sigma_p$, it follows that $M_V >_\ell M_{V'}$.   

Now suppose that no two paths in $V$ are ever more than a unit length apart, and let $i$ be the least index so that $V$ fails to be thin at the line $x = i$.  Then there are paths $\sigma_p$ and $\sigma_q$ of $V$ such that every east step of $\sigma_p$ to the left of $x = i$ is on or above the corresponding east step of $\sigma_q$ (though the two do not always coincide), but the $(i+1)^{st}$ step of $\sigma_q$ is a unit length above that of $\sigma_p$.  It is clear that $d_i(\sigma_p) < d_i(\sigma_q)$.  Let $j$ be the least index greater than $i$ such that $d_j(\sigma_p) > d_j(\sigma_q)$, let $\sigma_p'$ be the path obtained from $\sigma_p$ by deleting a north step along $x = j$ and adding one along $x = i$, and let $\sigma_q'$ be the path obtained from $\sigma_q$ by deleting a north step along $x = i$ and adding one along $x = j$.  The same argument from the first paragraph of this proof shows that both $\sigma_p'$ and $\sigma_q'$ are paths in $\Gamma(\alpha, \omega)$.  Again, let $V'$ be the vertex resulting from this symmetric exchange.  Since every east step of $\sigma_p$ in between $x = i$ and $x = j$ is exactly a unit length above the corresponding east step of $\sigma_q$, it follows that $\ell(M_V) = \ell(M_{V'})$.  Writing $>_{lex}$ for lexicographic order, we have the following chain:
\[
\ell(\sigma_p') >_{lex} \ell(\sigma_p) >_{lex} \ell(\sigma_q) >_{lex} \ell(\sigma_q').
\]
Thus $M_V >_\ell M_{V'}$. 
\end{proof}

\begin{lemma}\label{uniquethin}
Let $\mu$ be a monomial so that $\G(\mu)$ is nonempty.  Then $\G(\mu)$ has exactly one thin vertex.
\end{lemma}

\begin{proof}
Existence follows from Proposition \ref{switch} and the easy fact that a finite acyclic directed graph has at least one sink.  

To prove uniqueness, let $V = \{\sigma_1, \sigma_2, \ldots, \sigma_t\}$ be a thin vertex, ordered so that $\sigma_1 \succeq \sigma_2 \succeq \cdots \succeq \sigma_t$, and suppose $V$ is uniquely determined to the left of the line $x = i$ (where we allow $i = 0$).  Since $V$ is thin, there is an index $k$ and a number $p$ so that the $i^{th}$ east steps of the paths $\sigma_1, \sigma_2, \ldots, \sigma_k$ coincide and lie on the line $y = p$ and the $i^{th}$ east steps of the paths $\sigma_{k+1}, \sigma_{k+2}, \ldots, \sigma_t$ coincide and lie on the line $y = p-1$.  Now write $d_i(\mu) = qt + r$, with $r < t$.  

If $r \leq t-k$, then the paths $\sigma_{k+1}, \sigma_{k+2}, \ldots, \sigma_{k+r}$ must each have $q+1$ steps along the line $x = i$, while the rest must have $q$ north steps along this line.  If $r > t-k$, then each of the paths $\sigma_1, \sigma_2, \ldots, \sigma_{r - t + k}$ and $\sigma_{k+1}, \sigma_{k+2}, \ldots, \sigma_{t}$ must have $q+1$ steps along $x = i$, and the rest must have $q$ steps.  

Thus, $V$ is uniquely determined to the left of the line $x = i +1$, and the result follows. 
\end{proof}

\begin{proof}[Proof of Theorem \ref{main}]
Proposition \ref{switch} and Lemma \ref{uniquethin} imply that any $\G^\ell(\mu)$ has a unique sink (namely its thin vertex), so Theorem \ref{Grobnerswap} finishes the proof.  
\end{proof}

\ssection{Lattice path matroids}

The goal of this section is to prove the following analogue of Theorem \ref{main} for lattice path matroids.  

\begin{theorem}
Let $\mathcal{M} = \mathcal{M}(\alpha, \omega)$ be a lattice path matroid.  Then the toric ideal of $\mathcal{M}$ is generated by symmetric exchange binomials.  Moreover, there exists a monomial order on the base ring of $\mathcal{M}$ under which the symmetric exchange binomials form a Gr\"obner basis for the toric ideal.
\end{theorem}

\begin{proof}
Let $\sigma$ be a lattice path to the point $(n, r)$ with $N(\sigma) = \{a_1, a_2, \ldots, a_r\}$, where $a_1 < a_2 < \cdots < a_r$.  Define a lattice path $\overline{\sigma}$ to the point $(n+r, r)$ by $N(\overline{\sigma}) = \{a_1+1, a_2+2,  \ldots, a_r+r\}$, and note that $m(\overline{\sigma}) = x_{a_1}x_{a_2}\cdots x_{a_r}$.  

\begin{figure}[htp]
\centering
\includegraphics[height = 1.1in]{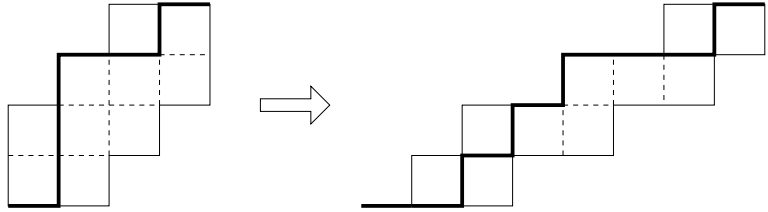}
\caption{A lattice path matroid $\mathcal{M}(\alpha, \omega)$ and the associated polymatroid $\Gamma(\overline{\alpha}, \overline{\omega})$.  If $\sigma$ is the bold path, note that $N(\sigma) = \{2, 3, 4, 7\}$ and $m(\overline{\sigma}) = x_2x_3x_4x_7$.}\label{path}
\end{figure}

Define a function from $\mathcal{M} = \mathcal{M}(\alpha, \omega)$ to $\Gamma = \Gamma(\overline{\alpha}, \overline{\omega})$ by $\sigma\rightarrow \overline{\sigma}$, and note that a lattice path $\sigma$ in between $\overline{\alpha}$ and $\overline{\omega}$ is in the image of this map if and only if it has no more than one north step along every line $x = i$, which is equivalent to $m(\sigma)$ being squarefree.

For a vertex $V = \{\sigma_1, \sigma_2, \ldots, \sigma_t\}$ of some $\G_{\mathcal{M}}(\mu)$, let $\overline{V} = \{\overline{\sigma_1}, \overline{\sigma_2}, \ldots, \overline{\sigma_t}\}$ denote the corresponding vertex of $\G_\Gamma(\mu')$ for some monomial $\mu'$.  We can now define a monomial order $>_{L}$ on the base ring of $\mathcal{M}$ by 
\[
M_V >_{L} M_{V'} \Leftrightarrow M_{\overline{V}} >_\ell M_{\overline{V'}}.
\]
Thus, the graph $\G_{\mathcal{M}}^{L}(\mu)$ is a directed subgraph of $\G_\Gamma^{\ell}(\mu')$ for some monomial $\mu'$.  Because a symmetric exchange between two squarefree monomials results in two squarefree monomials, Proposition \ref{switch} and Lemma \ref{uniquethin} imply that each directed graph $\G_{\mathcal{M}}^{L}(\mu)$ has a unique sink, and we can apply Theorem \ref{Grobnerswap}. 
\end{proof}

\noindent
\textbf{Acknowledgements.}  Thanks to Craig Huneke and Joe Bonin for many insightful conversations.



\bibliography{mybib}{}
\bibliographystyle{plain}

\end{document}